\documentclass[a4paper,11pt]{article}

\usepackage{graphicx}
\usepackage{amsmath}
\usepackage{amsfonts}
\usepackage{titlesec}
\usepackage{amsthm}
\usepackage{amssymb}
\usepackage{color}
\usepackage{mathtools}
\usepackage{geometry}
\usepackage{multirow}

\usepackage{setspace}
\usepackage{times}

\usepackage{algorithm}
\usepackage{algorithmicx, algpseudocode}

\usepackage{verbatim}

\numberwithin{equation}{section}

\newtheorem{theorem}{Theorem}[section]

\theoremstyle{definition}

\newtheorem{remark}[theorem]{Remark}

\definecolor{light-gray}{gray}{0.69}
\definecolor{light-red}{rgb}{1.0,0.4,0.4}

\definecolor{light-blue}{rgb}{0.4,0.45,1}
\definecolor{light-green}{rgb}{0.5,0.8,0.0}
\definecolor{dark-green}{rgb}{0.0,0.4,0.0}
\definecolor{dark-red}{rgb}{1.0,0.3,0.3}

\definecolor{dark-gray}{gray}{0.59}
\definecolor{very-dark-gray}{gray}{0.39}
\definecolor{lighter-red}{rgb}{1.0,0.6,0.6}

\definecolor{ocker_hell}{rgb}{0.75,0.7,0.4}
\definecolor{gelb_dunkel}{rgb}{0.75,0.7,0.0}
\definecolor{gruen_hell}{rgb}{0.5,0.8,0.0}

\definecolor{dark-blue}{rgb}{0.0,0.0,0.5}
\definecolor{new-blue}{rgb}{0.0,0.0,0.8}

\definecolor{lila}{rgb}{0.5,0.0,0.5}
\definecolor{dark-red}{rgb}{0.5,0.0,0.0}

\begin{document}

\renewcommand{\thefootnote}{\fnsymbol{footnote}}\setcounter{footnote}{0}

\begin{center}
{\Large Locally Conservative Continuous Galerkin FEM for Pressure Equation in Two-Phase Flow Model in Subsurfaces\footnote{This work was partially supported by a grant from the School of Energy Resources, University of Wyoming.} }
\end{center}

\renewcommand{\thefootnote}{\fnsymbol{footnote}}
\renewcommand{\thefootnote}{\arabic{footnote}}

\begin{center}
Q. Deng  and V. Ginting \\
Department of Mathematics, University of Wyoming, Laramie, Wyoming 82071, USA
\end{center}

\renewcommand{\baselinestretch}{1.5}
\begin{abstract}
A typical two-phase model for subsurface flow couples the Darcy equation for pressure and a transport equation for saturation in a nonlinear manner. In this paper, we study a combined method consisting of continuous Galerkin finite element methods (CGFEMs) followed by a post-processing technique for Darcy equation and finite volume method (FVM) with upwind schemes for the saturation transport equation, in which the coupled nonlinear problem is solved in the framework of operator decomposition. The post-processing technique is applied to CGFEM solutions to obtain locally conservative fluxes which ensures accuracy and robustness of the FVM solver for the saturation transport equation. We applied both upwind scheme and upwind scheme with slope limiter for FVM on triangular meshes in order to eliminate the non-physical oscillations. Various numerical examples are presented to demonstrate the performance of the overall methodology.
\end{abstract}



\section{Introduction} \label{sec:intro}

\renewcommand{\baselinestretch}{1.5}
The standard continuous Galerkin finite element methods (CGFEMs) are widely used for solving various kinds of partial differential equations \cite{brenner2008mathematical, thomas2013numerical} but are rarely utilized in two-phase flow simulation, which is due to the lack of local conservation property of the CGFEM solutions \cite{bush2013application, zhang2013locally}. There are application problems, such as multiphase flow (MPF) in fluid dynamics \cite{leveque2002finite} and drift-diffusion in electrodynamics \cite{deng2015construction, taflove2000computational}, which require local conservation property on their numerical solutions.  
\begin{figure}[ht]
\centering
\includegraphics[height=4.0cm]{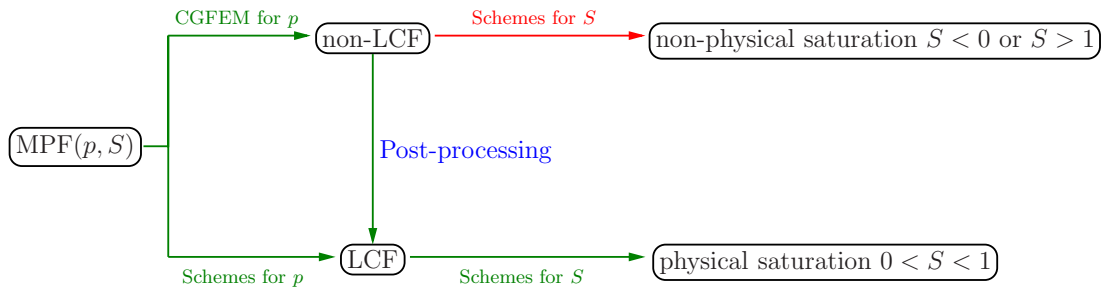} 
\caption{Illustration of why local conservation is required. LCF refers to locally conservative fluxes, $S$ refers to saturation, and $p$ refers to pressure.}
\label{fig:why}
\end{figure}
A simple illustration of why local conservation property is required in multiphase flow simulation is shown in Figure \ref{fig:why}. Instead of using the standard CGFEMs for such application problems, many endeavors have utilized mixed finite element methods (MFEMs) (see for example \cite{chavent1991unified, cui2007combined, kou2014analysis, murad2013new, ewing1983mathematics, russell1983finite, wang2000approximation}), FVM (see for example \cite{forsyth1991control, michel2003finite, monteagudo2004control, reichenberger2006mixed}), and discontinuous Galerkin (DG) methods (see for example \cite{epshteyn2007fully, epshteyn2009analysis, fuvcik2011discontinous, hoteit2008efficient, nayagum2004modelling}).

The paper \cite{kou2014analysis} analyzed a combined method for incompressible two-phase flow in porous media. The method consists of the MFEM for Darcy equation and DG for the transport equation. Under appropriate conditions, the authors established existence and uniqueness of the numerical solutions by using a constructive approach as well as derived the optimal error estimates for saturation in $L^2(H^1)$. Wang and his collaborators developed an Eulerian-Lagrangian localized adjoint method (ELLAM) to discretize the transport component  of the miscible fluid flows in porous media \cite{wang2000approximation}, where the MFEM was used to approximate the pressure and Darcy velocity.

A finite volume scheme for two-phase immiscible flow in porous media is presented and analyzed in \cite{michel2003finite}. The method combines a centered finite volume scheme for the Chavent global pressure equation and an upwind time implicit finite volume scheme for the transport equation. The capillary pressure is included in the model and error estimates of the approximate solutions are established.  A fully coupled implicit DG is developed in \cite{epshteyn2007fully} for both the pressure and saturation equations for two-phase flow model in porous media. The proposed schemes do not require slope limiting or upwind stabilization techniques but pay a price of solving much larger linear algebra systems than that of CGFEM. The optimal convergence behavior and error estimates are studied in \cite{epshteyn2009analysis}.

In \cite{ginting2015application}, a fairly recently developed method called weak Galerkin finite element method (WGFEM), was applied to solve Darcy equation in two-phase model for subsurface flow. The numerical normal velocity obtained from WGFEM satisfies local conservation property, which ensures robustness of the finite volume solver for the transport saturation equation. The paper \cite{zhang2013locally} studied a locally conservative Galerkin (LCG) finite element method for two-phase flow simulations in heterogeneous porous media. The method utilizes the property of local conservation at steady state conditions to post-process the numerical fluxes at the element boundaries at every time step. Other numerical methods for two-phase flow simulations can be found for instance in \cite{du2010efficient, efendiev2012robust}.

In this paper, we apply the classical CGFEM supplied with a post-processing technique to the Darcy equation in the two-phase flow model. The nontrivial but simple post-processing technique developed in \cite{deng_construction} is applied to linear and quadratic CGFEM solutions to generate locally conservative fluxes, which will be utilized in the explicit FVM time discretization for solving the saturation transport equation. Similar post-processing techniques for linear CGFEM can be found for example in \cite{bush2015locally, bush2013application, cockburn2007locally, deng2015construction, deng_locally, gmeiner2014local, hughes2000continuous, nithiarasu2004simple, sun2009locally,  thomas2008element, thomas2008locally, zhang2013locally} and for high order CGFEM can be found for example in \cite{cockburn2007locally}. In order to eliminate the potential non-physical oscillations of the saturation profile at the front, we combine the FVM with upwind schemes. In this paper, we present upwind scheme and upwind with slope limiter for the triangular meshes.

The rest of this paper is organized as follows. Section \ref{sec:tpfm} sets up the two-phase model problem and Section \ref{sec:meth} outlines our numerical methodology for solving the two-phase flow problem that is based on operator splitting. Section \ref{sec:pp} presents the post-processing technique for CGFEMs to obtain locally conservative fluxes. Upwind scheme and upwind scheme with slope limiter for triangular meshes are described in Section \ref{sec:upwind}. Section \ref{sec:num} presents various numerical results on the application of the post-processing techniques and upwind and upwind with slope limiter for a couple of examples in single-phase flow and two-phase flow. Conclusion is given in Section \ref{sec:conclusion}.

\section{Two-Phase Flow Modelling} \label{sec:tpfm}
Two-phase flow is a particular example of multiphase flow typically occuring in an immiscible system containing a wetting phase and a non-wetting phase, for instance water and oil, with a meniscus separating the two phases.  We consider a heterogeneous oil reservoir which is confined in a domain $\Omega$. The governing equations consist of Darcy's law and a statement of conservation of mass and are expressed as
\begin{equation} \label{eq:tpf}
\begin{cases}
\nabla \cdot \boldsymbol{v} = q, \quad \text{where} \quad \boldsymbol{v} = - \lambda(S) \kappa (\boldsymbol{x}) \nabla p, \quad \text{in} \quad  \Omega \times (0, T], \\
\partial_t S + \nabla \cdot (f(S) \boldsymbol{v} ) = q_w, \quad \text{in} \quad \Omega \times (0, T], \\
p(\boldsymbol{x}) = g_D(\boldsymbol{x}), \quad \boldsymbol{x} \in \Gamma_D, \\
\boldsymbol{v} \cdot \boldsymbol{n} = g_N(\boldsymbol{x}), \quad \boldsymbol{x} \in \Gamma_N, \\
S(\boldsymbol{x}, 0) = S_0(\boldsymbol{x}), \quad \boldsymbol{x} \in \Omega,\\ 
\end{cases}
\end{equation}
where $\boldsymbol{v}$ is the Darcy velocity, $\kappa$ is the permeability coefficient, $S$ is the saturation of the wetting phase, $S_0$ is the initial saturation, $T$ is the finial time, and $\Gamma_D \cup \Gamma_N = \partial \Omega, \Gamma_D \cap \Gamma_N = \emptyset$. The total mobility $\lambda(S)$ and the flux function $f(S)$ are given by 
\begin{equation} \label{eq:tmff}
\lambda(S) = \frac{\kappa_{rw}(S)}{\mu_w} + \frac{\kappa_{ro}(S)}{\mu_o}, \quad f(S) = \frac{\kappa_{rw}(S)/\mu_w}{\lambda(S)},
\end{equation}
where $\kappa_{r\alpha}$ ($\alpha = w$ denoting the wetting phase and $\alpha = o$ denoting the non-wetting phase) is the relative permeability of the phase $\alpha$. For simplicity, here capillary pressure and gravity are not included in the model.

\section{Methodology}  \label{sec:meth}
\begin{figure}[ht]
\centering
\includegraphics[height=5.0cm]{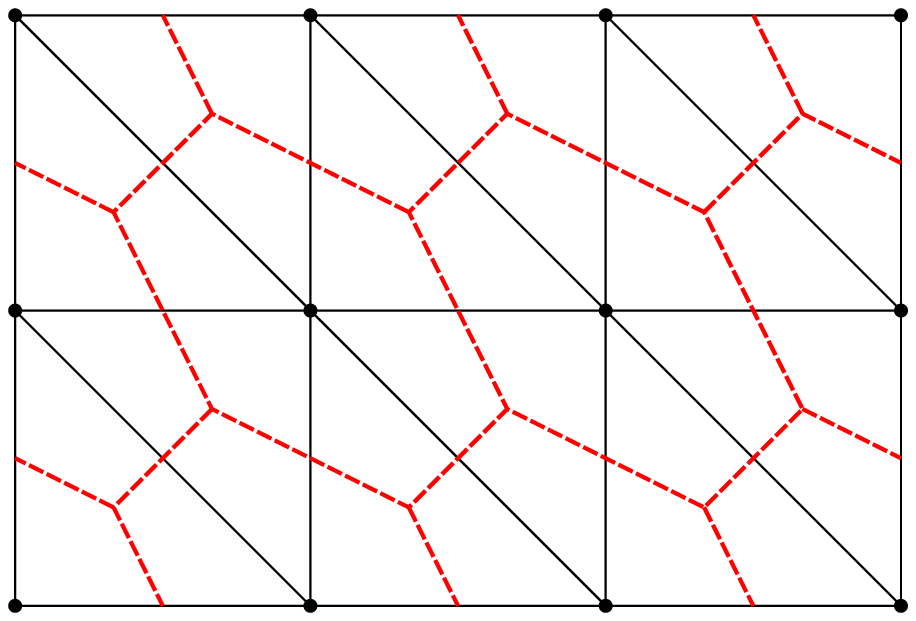}  \hspace*{0.45cm}
\includegraphics[height=5.0cm]{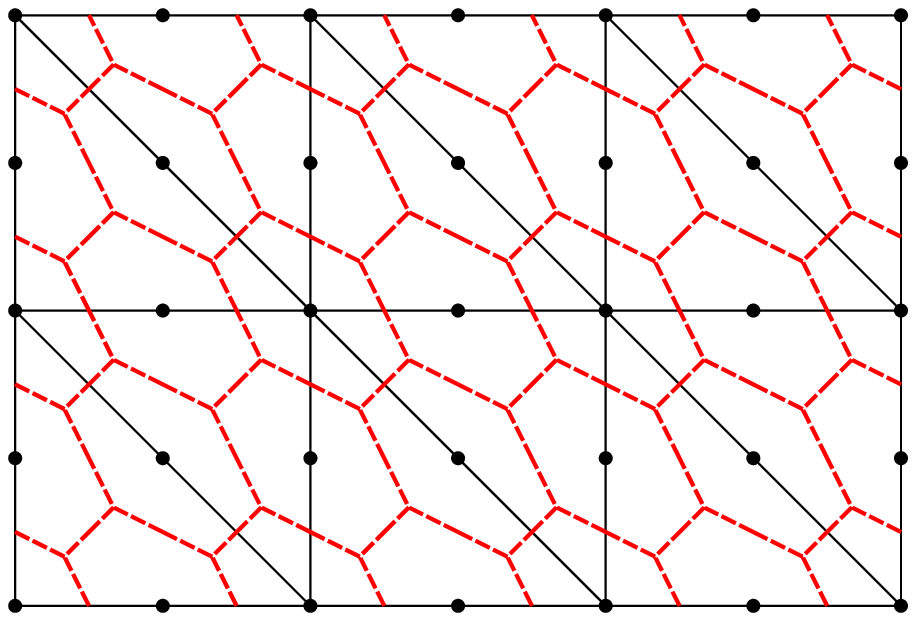}
\caption{Discretization of $\Omega$ into a collection of triangular elements $\mathcal{T}_h$ (black lines). The dots represent the nodal degrees of freedom for the approximated pressure using linear CGFEM (left plot) and quadratic CGFEM (right plot). The red dotted lines represent collection of nodal centered control volumes $\mathcal{D}_h$.}
\label{fig:meshes}
\end{figure}

In this section, we describe the methodology for solving the two-phase flow model problem. Let $\Omega$ be a bounded domain in $\mathbb{R}^2$ with Lipschitz boundary $\partial\Omega$. We consider a partition $\mathcal{T}_h$, a mesh for CGFEMs, consisting of non-overlapping triangular elements $\tau$ such that $\overline\Omega = \bigcup_{\tau\in\mathcal{T}_h} \tau$ and a dual mesh $\mathcal{D}_h$ consisting of control volumes $C^z$ ($z$ is an associated nodal degree of freedom) such that $\overline\Omega = \bigcup_{C^z \in\mathcal{D}_h} C^z$ (see Figure \ref{fig:meshes}). We denote $Z = Z_{\text{in}} \cup Z_\text{d}$, where $Z_{\text{in}}$ is the set of interior nodal degrees of freedom and $Z_\text{d}$ is the set of corresponding points on $\partial\Omega$. We set $h=\max_{\tau\in\mathcal{T}_h} h_\tau$ where $h_\tau$ is defined as the diameter of $\tau$.  We denote the standard CGFEM space as
\begin{equation*}
V^k_h = \big\{w_h\in C(\overline\Omega): w_h|_\tau \in P^k(\tau), \ \forall \ \tau\in\mathcal{T}_h \ \text{and} \ w_h|_{\Gamma_D} = 0 \big\}, 
\end{equation*} 
where $P^k(\tau)$ is a space of polynomials with degree at most $k$ on $\tau$. In this paper, we focus on the case where $k=1, 2$.  We also denote the space of piecewise constant functions on $\mathcal{D}_h$ by
\begin{equation*}
W^0_h = \big\{w_h: w_h|_{C^z} \in P^0(C^z), \ \forall \ C^z \in \mathcal{D}_h \big\}.
\end{equation*}
Then a semi-discretization of \eqref{eq:tpf} is to find $p_h(t) = p_h$ with $(p_h - g_{D,h}) \in V^k_h$ and $S_h(t) = S_h \in W^0_h $, such that
\begin{equation} \label{eq:meth} 
\begin{cases}
a(S_h; p_h, v_h) = \ell_p(v_h) \quad \forall \ v_h \in V^k_h,  \\
( \partial_t S_h, w_h ) + b(p_h; S_h, w_h) = \ell_S(w_h),  \quad \forall \ w_h \in W^0_h, \\
\end{cases}
\end{equation} 
where
\begin{equation*}
a(u; v, w) = \int_\Omega \lambda(u) \kappa  \nabla v \cdot \nabla w \ \text{d} \boldsymbol{x}, \quad b(p_h; v, w) = \sum_{C^z \in \mathcal{D}_h } \int_{\partial C^z} \boldsymbol{v}_h(p_h) \cdot \boldsymbol{n} f(v) w \ \text{d} l, 
\end{equation*}
\begin{equation*}
\quad \ell_p(w) = ( q,  w ) - \langle g_N, w \rangle_{\Gamma_N},  \quad  \ell_S(w) = ( q_w,  w ), 
\end{equation*}
$( \cdot, \cdot )$ is the usual $L^2$ inner product, $\langle \cdot, \cdot \rangle_{\Gamma}$ is the $L^2$ inner product on the curve $\Gamma$, and $g_{D,h} \in V_h^k$ can be thought of as the interpolant of $g_D$ using the usual finite element basis. Notice that here the first equation is a CGFEM formulation while the second equation is a FVM formulation.

The system \eqref{eq:meth} is coupled nonlinearly. There are several nontrivial issues inherent in this system that require appropriate attention and treatment. The main objective is to devise a numerical procedure that is stable and gives accurate approximate solutions. It should be noted however, practioners have developed legacy codes in a relatively independent manner for every equation in \eqref{eq:tpf}. From this vantage point, the best strategy is to construct a reliable framework that allows for taking advantage of the well established codes.

For the time discretization, we denote,  for the coarse time step discretization, a partition of $[0, T]$ as $0 = t_0 < t_1 < \cdots < t_n < \cdots < t_N = T$. Let $I_n = [t_{n-1}, t_n]$ and $\Delta t_n = t_n - t_{n-1}.$  On each coarse time step $I_n$, we denote, for the fine time step, a partition of $I_n$ as $t_{n-1} = t_{n_0} < t_{n_1} < \cdots <  t_{n_r} < \cdots < t_{n_R} = t_n$. Let $\Delta t_{n_r} = t_{n_r}  - t_{n_{r-1}} $ be the fine time step. Applying an explicit time stepping for the transport equation on fine time step interval 
$I_{n_r}$, we obtain 
\begin{equation} \label{eq:methsat} 
( S_h^{n_r}, w_h ) = ( S_h^{n_{r-1}}, w_h ) - \Delta t_{n_r} b(p_h; S_h^{n_{r-1}}, w_h) + \Delta t_{n_r}  \ell_S(w_h),  \quad \forall \ w_h \in W^0_h.
\end{equation} 

\noindent
To summarize, we present the following algorithm for the overall methodology on one time step.
\begin{algorithm} [H] 
\caption{Methodology for Time Marching on $(t_{n-1}, t_n]$}
\label{alg:timemarch}
\begin{algorithmic}
\State Let $S_h^{n-1}, p_h^{n-1}$ be available.
%
\State Set $S^{n, 0}_h = S_h^{n-1}$.
\State Loop $m = 1, 2, \cdots, M_n$ for iteration \\
\quad Find $p^{n, m}_h$ satisfying
\begin{equation}
a(S^{n, m-1}_h; p_h^{n,m}, v_h) = \ell_p(v_h), \quad \forall \ v_h \in V^k_h.
\end{equation}
\State \quad Post-process $p_h^{n,m}$ to be $\widetilde p_h^{n,m}$ to obtain locally conservative Darcy velocity. 
\State \quad Set $S_h^{n_0, m} = S^{n, m-1}_h$.
\State \quad Loop $r = 1, 2, \cdots, R$ for the fine time step
\State \qquad Find $S_h^{n_r, m}$ satisfying
\begin{equation} \label{eq:upwindsat}
( S_h^{n_r, m}, w_h ) = ( S_h^{n_{r-1}, m}, w_h ) - \Delta t_{n_r} b(\widetilde p_h^{n,m}; S_h^{n_{r-1}, m}, w_h) + \Delta t_{n_r}  \ell_S(w_h),  \quad \forall \ w_h \in W^0_h.
\end{equation} 
\State \quad End loop

\State \quad Set $S^{n, m}_h = S_h^{n_R, m}$
\State End loop
\State Set $S_h^n = S^{n, M_n}_h, p_h^n = p_h^{n,M_n}$.

\end{algorithmic}
\end{algorithm}

\section{A Post-processing Technique}  \label{sec:pp}

The numerical Darcy velocity $\boldsymbol{v}$ directly calculated from the CGFEM solution does not satisfy local conservation and are not continuous across the element interfaces, which are desired properties for solving the transport equation. We present in this section a post-processing technique developed in \cite{deng_construction} for the CGFEM solutions that aims at producing locally conservative flux.

\subsection{Auxiliary Elemental Problem} \label{sec:bvp}

We would like to have Darcy velocity to satisfy local conservation on control volumes which are generated from the dual mesh $\mathcal{D}_h$ whose normal component is continuous at the boundaries of these control volumes (see Figure \ref{fig:meshes}). The construction of the dual mesh on a single element is shown in Figure \ref{fig:dualelem}. In details, for linear CGFEM, we connect the barycenter of the triangle to the middle points of the edges of the triangle, and for quadratic CGFEM, we divide the triangle into four sub-triangles and then do the same for each sub-triangle as for the linear case. Each control volume corresponds to a degree of freedom in CGFEMs. 

For simplicity, we denote $K = \lambda(S) \kappa (\boldsymbol{x})$ in \eqref{eq:tpf}. We post-process the pressure solution $p_h$  solved by CGFEM to obtain 
$\widetilde {\boldsymbol{\nu}}_h = - K \nabla \widetilde{p}_h $ such that it is continuous at the boundaries of each control volume and satisfies the local conservation property in the sense
\begin{equation} \label{eq:cvconservation}
\int_{\partial C^\xi} \widetilde {\boldsymbol{\nu}}_h \cdot \boldsymbol{n}  \ \text{d} l = \int_{C^\xi} q \ \text{d} \boldsymbol{x},
\end{equation}
where $C^\xi$ is a control volume surrounding a node $\xi$. This node is associated with a degree of freedom in $V^k_h.$

\begin{figure}[ht]
\centering
\includegraphics[height=5.0cm]{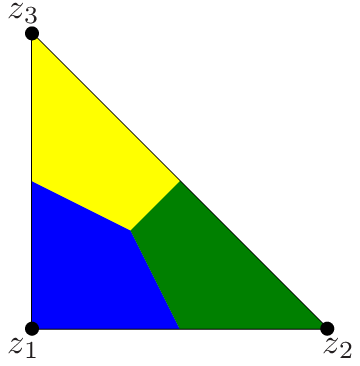}  \hspace*{1.85cm}
\includegraphics[height=5.0cm]{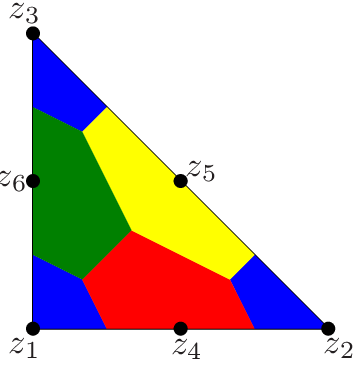}
\caption{Construction of nodal based control volumes for $V^1_h$ (left) and $V^2_h$ (right).}
\label{fig:dualelem}
\end{figure}

In order to obtain locally conservative fluxes on each control volume, we set and solve an elemental/local problem on $\tau$. Let $N_k = \frac{1}{2}(k+1)(k+2)$ be the total number of degrees of freedom on a triangular element for $V_h^k$. We denote the collection of those degrees of freedom by $s(\tau, k) = \{ z_j \}_{j=1}^{N_k}$ and partition each element $\tau$ into $N_k$ non-overlapping polygonals $\{ t_{z_j} \}_{j=1}^{N_k}$; see Figure \ref{fig:dualelem}.
For $t_\xi$ with $\xi \in s(\tau,k)$, we make a decomposition $\partial t_\xi = ( \partial \tau \cap \partial t_\xi ) \cup ( \partial C^\xi \cap \partial t_\xi ).$  We also define the average on an edge or part of the edge sharing by two elements $\tau_1$ and $\tau_2$ for vector $\boldsymbol{v}$ as
\begin{equation} \label{def:ave}
\{ \boldsymbol{v} \} = \frac{\boldsymbol{v}|_{\tau_1} + \boldsymbol{v}|_{\tau_2} }{2}. 
\end{equation}

Let $W^0(\tau)$ be the space of piecewise constant functions on element $\tau$ such that $W^0(\tau) = \text{span} \{ \psi_\eta \}_{\eta \in s(\tau,k)}$, where $\psi_\eta$ is the characteristic function of the polygonal $t_\eta$, i.e.,
\begin{equation}
\psi_\eta(\boldsymbol{x}) = \Big\lbrace\begin{array}{c} 1 \quad \text{if} \quad \boldsymbol{x} \in t_\eta  \\ 0 \quad \text{if} \quad \boldsymbol{x} \notin t_\eta \end{array}.
\end{equation}
We define a map $I_\tau: H^1(\tau) \rightarrow W^0(\tau)$ with $I_\tau w = \displaystyle \sum_{\xi \in s(\tau,k)} w_\xi \psi_\xi$, where $w_\xi = w(\xi)$ for $w\in H^1(\tau)$. We also define the following bilinear forms
\begin{equation} \label{eq:bebiforms}
b_\tau( v, w)  = -\sum_{\xi \in s(\tau,k)} \int_{\partial C^\xi \cap \partial t_\xi} K  \nabla v  \cdot \boldsymbol{n} I_\tau w \ \text{d} l, \qquad e_\tau(v, w) = \int_{\partial \tau} \{ K \nabla v \} \cdot \boldsymbol{n}  w  \ \text{d} l.
\end{equation}

Let $ V^k_h(\tau) = \text{span}\{  \phi_\eta \}_{ \eta \in s(\tau,k) }$ where $ \phi_\eta$ can be thought as the usual
nodal $\eta$ basis function restricted to element $\tau$.
The elemental calculation for the post-processing is to find $\widetilde p_{\tau, h} \in V^k_h(\tau) $ satisfying 
\begin{equation} \label{eq:bvpvf}
b_\tau(\widetilde p_{\tau, h}, w) = \ell_\tau ( I_\tau w - w ) + a_\tau(p_h, w) + e_\tau(p_h, I_\tau w - w), \quad \forall \ w \in V^k_h(\tau),
\end{equation}
where $\displaystyle{a_\tau (w, v) = \int_\tau K\nabla w \cdot \nabla v  \ \text{d} \boldsymbol{x} }.$

\newtheorem{lem}{Lemma}[section]
\begin{lem}  \label{lem:ebvp}
The variational formulation \eqref{eq:bvpvf} has a unique solution up to a constant.
\end{lem}

\begin{proof}
We have $V_h^k(\tau) = \text{span}\{  \phi_\xi \}_{ \xi \in s(\tau,k) }$, where $\phi_\xi(\eta) = \delta_{\xi \eta}$ with $\delta_{\xi \eta}$ being the Kronecker delta,
for all $\xi, \eta \in s(\tau,k)$. By replacing the test function $w$ with  $\phi_\xi$ for all $\xi \in  s(\tau,k)$,
 \eqref{eq:bvpvf}  is reduced to 
\begin{equation}  \label{eq:cz10}
- \int_{\partial C^\xi \cap \partial t_\xi } K  \nabla \widetilde p_{\tau, h} \cdot \boldsymbol{n}  \ \text{d} l = \int_{t_\xi} q \ \text{d} \boldsymbol{x} - \ell_\tau (  \phi_\xi ) + a_\tau(p_h,  \phi_\xi ) + e_\tau( p_h, I_\tau  \phi_\xi -  \phi_\xi ), \quad \forall \ \xi \in  s(\tau,k).
\end{equation} 
This is a fully Neumann boundary value problem in $\tau$ with boundary condition satisfying 
\begin{equation} \label{eq:bvpb} 
- \int_{\partial \tau \cap \partial t_\xi } K  \nabla \widetilde p_{\tau, h}  \cdot \boldsymbol{n}  \ \text{d} l = \ell_\tau (  \phi_\xi ) - a_\tau(p_h,  \phi_\xi ) - e_\tau( p_h, I_\tau  \phi_\xi -  \phi_\xi ), \quad \forall \ \xi \in s(\tau,k).
\end{equation} 
To establish the existence and uniqueness of the solution, one needs to verify the compatibility condition \cite{evans2010partial}. We calculate 
\begin{equation*}
-\int_{\partial \tau} K  \nabla \widetilde p_{\tau, h} \cdot \boldsymbol{n} \ \text{d} l = \sum_{\xi \in s(\tau,k)} \Big( \ell_\tau (  \phi_\xi ) - a_\tau(p_h,  \phi_\xi ) - e_\tau( p_h, I_\tau  \phi_\xi -  \phi_\xi ) \Big).
\end{equation*}
Using the fact that $\sum_{\xi \in s(\tau,k)}   \phi_\xi = 1$ and linearity, we obtain
\begin{equation*}
\sum_{\xi\in s(\tau,k)} \ell_\tau(  \phi_\xi) = \sum_{\xi\in s(\tau,k)} \int_\tau q    \phi_\xi \ \text{d} \boldsymbol{x} = \int_\tau q \sum_{\xi\in s(\tau,k)}   \phi_\xi \ \text{d} \boldsymbol{x} = \int_\tau q \ \text{d} \boldsymbol{x}.
\end{equation*}
Using the fact that  $\nabla \big( \sum_{\xi\in s(\tau,k)}   \phi_\xi \big) = \boldsymbol{0} $ and linearity, we obtain
\begin{equation*}
\sum_{\xi\in s(\tau,k)} a_\tau(p_h,   \phi_\xi) = \sum_{\xi\in s(\tau,k)} \int_\tau K  \nabla p_h \cdot \nabla   \phi_\xi \ \text{d} \boldsymbol{x} = \int_\tau K  \nabla p_h \cdot \nabla \big( \sum_{\xi\in s(\tau,k)}   \phi_\xi \big) \text{d} \boldsymbol{x} = 0.
\end{equation*}
Also, we notice that 
\begin{equation*}
\sum_{\xi\in s(\tau,k)} e_\tau(p_h,   I_\tau  \phi_\xi -  \phi_\xi ) = \sum_{\xi\in s(\tau,k)}  \int_{\partial \tau \cap \partial t_\xi } \{ K  \nabla p_h \} \cdot \boldsymbol{n}  \ \text{d} l - \int_{\partial \tau} \{ K  \nabla p_h \} \cdot \boldsymbol{n} \sum_{\xi\in s(\tau,k)}   \phi_\xi \ \text{d} l = 0.
\end{equation*}
Combining these equalities, compatibility condition $\displaystyle{\int_{\partial \tau} - K  \nabla \widetilde p_{\tau, h} \cdot \boldsymbol{n} \ \text{d} l = \int_\tau q \ \text{d} \boldsymbol{x}}$ is verified. 
This completes the proof.
\end{proof}

\begin{remark}
The technique proposed here can be naturally generalized to rectangular elements. 
\end{remark}

\subsection{Elemental Linear System } \label{sec:llas}
We note that the dimension of $V^k_h(\tau)$ is $N_k$ and hence the variational formulation \eqref{eq:bvpvf} yields an $N_k$-by-$N_k$ linear algebra system. To see this, we represent $\widetilde p_{\tau, h} \in V^k_h(\tau)$ as
\begin{equation} \label{eq:ppsol}
\widetilde p_{\tau, h} = \sum_{\eta \in s(\tau,k)} \alpha_\eta   \phi_\eta
\end{equation}
in \eqref{eq:bvpvf} and replace the test function
by $\phi_\xi$ to yield
\begin{equation} \label{eq:axb}
A \boldsymbol{\alpha} = \boldsymbol{\beta},
\end{equation}
where 
$\boldsymbol{\alpha} \in \mathbb{R}^{N_k} $ whose entries are the nodal solutions in \eqref{eq:ppsol}, $\boldsymbol{\beta} \in \mathbb{R}^{N_k} $ with entries
\begin{equation}
\beta_\xi = \ell_\tau ( I_\tau \phi_\xi - \phi_\xi ) + a_\tau(p_h,   \phi_\xi) + e_\tau (p_h,   I_\tau \phi_\xi - \phi_\xi ), \quad \forall \ \xi \in s(\tau,k),
\end{equation}
and
\begin{equation}
A_{\xi \eta} = b_\tau( \phi_\eta,   \phi_\xi), \quad \forall \ \xi, \eta \in s(\tau,k).
\end{equation}

The linear system \eqref{eq:axb} is singular and there are infinitely many solutions since the solution to \eqref{eq:bvpvf} is unique up to a constant by Lemma \ref{lem:ebvp}. However, this does not cause any issue since
 to obtain locally conservative fluxes, the desired quantity from the post-processing
is $\nabla \widetilde{p}_{\tau,h}$, which is unique.

\subsection{Local Conservation} \label{sec:loccons}
 
At this stage, we verify the local conservation property \eqref{eq:cvconservation} on control volumes for the post-processed solution. It is stated in the following lemma.

\begin{lem}  \label{lem:localconserv}
The desired local conservation property \eqref{eq:cvconservation} is satisfied on the control volume $C^\xi$ where $\xi \in Z_{\text{in}}$.
\end{lem}

\begin{proof}
For a basis function $\phi_\xi$, let $\Omega^\xi = \cup_{i=1}^{N_\xi} \tau_i^\xi$ be its support. Noting that the gradient component is averaged, it is obvious that 
\begin{equation*}
\sum_{j=1}^{N_\xi}  \int_{\partial \tau_j^\xi} \{ K  \nabla p_{\tau_j, h} \} \cdot \boldsymbol{n} \phi_\xi  \ \text{d} l = 0 \qquad \text{and} \qquad \sum_{j=1}^{N_\xi} \int_{\partial \tau_j^\xi \cap \partial t_\xi } \{ K  \nabla p_{\tau_j, h} \} \cdot \boldsymbol{n}  \ \text{d} l = 0.
\end{equation*}
This implies that $\sum_{j=1}^{N_\xi} e_{\tau_j^\xi}(p_h, \phi_\xi) = 0.$ 
We notice that $\phi_\xi$ has a support on $N_\xi$ element, which yields
\begin{equation}
\sum_{j=1}^{N_\xi} \Big( a_{\tau_j^\xi}(p_h, \phi_\xi) - \ell_{\tau_j^\xi}(\phi_\xi) \Big) = 0
\end{equation}
Using this equality, straightforward calculation gives
\begin{equation*}
 \int_{ \partial C^\xi } - K  \nabla \widetilde p_{\tau, h} \cdot \boldsymbol{n} \ \text{d} l =  \int_{C^\xi} q \ \text{d} \boldsymbol{x} + \sum_{j=1}^{N_\xi} \Big( a_{\tau_j^\xi}(p_h, \phi_\xi) - \ell_{\tau_j^\xi}(\phi_\xi) - e_{\tau_j^\xi}(p_h, \phi_\xi) \Big) = \int_{C^\xi} q \ \text{d} \boldsymbol{x},
\end{equation*}
which completes the proof.
\end{proof}

\section{Upwind Schemes on Triangular Elements}  \label{sec:upwind}
In this section, we present the well-studied upwind scheme for solving the transport equation for saturation. The scheme is widely used on both rectangular and triangular meshes (see for example \cite{bush2013application, ginting2015application, thomas2013numerical}). It is known that (see for example \cite{bush2013application}) this scheme is stable and preserve positivity of the solution. However, it has the drawback of smearing shock / front feature of the solution. To alleviate this, we adopt a slope limiter procedure (see for example \cite{thomas2013numerical}). Application of limiters in upwind scheme has been relatively limited to rectangular meshes. In this section, we propose an upwind scheme with slope limiter for triangular meshes.
\begin{figure}[ht]
\centering
\includegraphics[height=7.0cm]{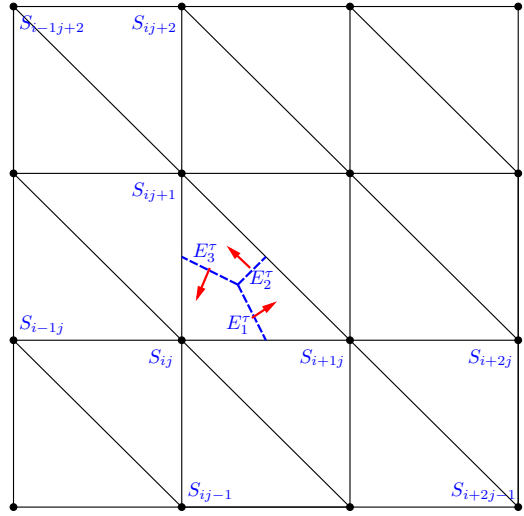} 
\caption{An example element to illustrate the upwind and upwind with slope limiter.}
\label{fig:upwind}
\end{figure}

The transport equation governing the saturation is solved by FVM with explicit time stepping \eqref{eq:upwindsat} combined with an upwind scheme for the integration term $\displaystyle{\int_{\partial C^z} \widetilde{\boldsymbol{v}}_h \cdot \boldsymbol{n} f(S_h) \ \text{d} l}.$ The upwind scheme is an approximation technique for this integration, which is aimed to stabilize the scheme and eliminate the potential non-physical oscillations of the saturation profile at the front. A review of upwind schemes on a rectangular mesh can be found in \cite{thomas2013numerical}. We now use Figure \ref{fig:upwind} to illustrate the upwind schemes on triangular meshes for linear CGFEM. For quadratic CGFEM, this can be done similarly.

For simplicity, we assume $f'(S) \geq 0$. In Figure \ref{fig:upwind}, taking the edge $E^\tau_1$ for example, to approximate $\int_{\partial E^\tau_1} \widetilde{\boldsymbol{v}}_h \cdot \boldsymbol{n} f(S_h) \ \text{d} l$, the upwind scheme is 
\begin{equation} \label{eq:firstorderupwind}
\int_{\partial E^\tau_1} \widetilde{\boldsymbol{v}}_h \cdot \boldsymbol{n} f(S_h) \ \text{d} l
\approx
\begin{cases}
\displaystyle{ f(S_{ij}) \int_{\partial E^\tau_1} \widetilde{\boldsymbol{v}}_h \cdot \boldsymbol{n} \ \text{d} l, \quad \text{if} \quad \int_{\partial E^\tau_1} \widetilde{\boldsymbol{v}}_h \cdot \boldsymbol{n} \ \text{d} l > 0, } \\
\displaystyle{ f(S_{i+1j}) \int_{\partial E^\tau_1} \widetilde{\boldsymbol{v}}_h \cdot \boldsymbol{n} \ \text{d} l, \quad \text{otherwise},}
\end{cases}
\end{equation}
while upwind scheme with slope limiter is
\begin{equation} \label{eq:secondorderupwind}
\int_{\partial E^\tau_1} \widetilde{\boldsymbol{v}}_h \cdot \boldsymbol{n} f(S_h) \ \text{d} l
\approx
\begin{cases}
\displaystyle{
f(\widetilde S_{ij}) \int_{\partial E^\tau_1} \widetilde{\boldsymbol{v}}_h \cdot \boldsymbol{n} \ \text{d} l, \quad \text{if} \quad \int_{\partial E^\tau_1} \widetilde{\boldsymbol{v}}_h \cdot \boldsymbol{n} \ \text{d} l > 0, } \\
\displaystyle {f(\widetilde S_{i+1j}) \int_{\partial E^\tau_1} \widetilde{\boldsymbol{v}}_h \cdot \boldsymbol{n} \ \text{d} l, \quad \text{otherwise},}
\end{cases}
\end{equation}
where 
\begin{equation} \label{eq:secondsat}
\begin{aligned}
\widetilde S_{ij} & = S_{ij} - \frac{\min \{ | S_{i-1j} - S_{ij} |, | S_{ij} - S_{i+1j} | \} }{2}, \\
\widetilde S_{i+1j} & = S_{i+1j} - \frac{\min \{ | S_{ij} - S_{i+1j} |, | S_{i+1j} - S_{i+2j} | \} }{2}. \\
\end{aligned}
\end{equation}

\section{Numerical Experiments}\label{sec:num}
In this section, we present various examples to illustrate the performance of the post-processing techniques applied to linear and quadratic CGFEMs.  We consider two different simulation scenarios, single phase flow simulation and two-phase flow simulation. Both simulations use Algorithm \ref{alg:timemarch} with only one iteration ($M_n=1$) that is equivalent to the usual operator splitting scheme commonly called Implicit Pressure Explicit Saturation (IPES). For both scenarios, we present errors in $L^2$-norm of the saturation profile at the final time. For the single phase flow, we present the local conservation errors (LCEs) of both CGFEM solution $p_h$ and the post-processed solution $\widetilde p_h$ as well as the $H^1$ semi-norm errors of $\widetilde p_h$. 

In all these examples, the unit domain $\Omega=[0, 1]^2$ is discretized uniformly for simplicity, namely, divide it into $N\times N$ rectangles and each rectangle is divided into two tirangles. For pressure, we consider homogeneous Neumann boundary conditions at the top and bottom of the domain and Dirichlet boundary conditions $p=1$ and $p=0$ at the left and right boundaries, respectively. We define permeability functions  
\begin{equation} \label{eq:ex11perm}
\kappa(\boldsymbol{x}) = \frac{1}{1-0.8\sin(6\pi x_1)} \cdot \frac{1}{1-0.8\sin(6\pi x_2)},
\end{equation}
\begin{equation} \label{eq:ex12perm}
\kappa(\boldsymbol{x}) = \frac{1}{0.25-0.999(x_1-x_1^2)\sin(11.2 \pi x_1)} \cdot \frac{1}{0.25-0.999(x_2-x_2^2)\cos(5.2 \pi x_2)},
\end{equation}
\begin{equation} \label{eq:ex13perm}
\kappa(\boldsymbol{x}) = \frac{e^{1-x_1}(x_2-x_2^2)}{x_1+1},
\end{equation}
flux function
\begin{equation} \label{eq:exff}
f(S) = \frac{S^2}{S^2/1 + (1-S)^2/5},
\end{equation}
and initial saturations
\begin{equation} \label{eq:ex12is}
S_0(\boldsymbol{x}) = 
\begin{cases}
1, & x_1 \leq 0, \\
0, & x_1 > 0,
\end{cases}
\end{equation} 
\begin{equation} \label{eq:ex11is}
S_0(\boldsymbol{x}) = 
\begin{cases}
1, & x_1 < 0, \\
\frac{1}{1+  x_1^2}, & x_1 \geq 0.
\end{cases}
\end{equation}

\subsection{Single Phase Flow} \label{sec:spf}
Single phase flow occurs when the pores are completely filled with only one fluid phase, which in the context of this paper is water. The two-phase flow model is naturally reduced to a model of single phase flow when we assume $\lambda(S)=1$ in \eqref{eq:tpf}. We consider the following examples.

\textbf{Example 1-1.} $\kappa(\boldsymbol{x})$ is \eqref{eq:ex11perm}, $f(S)$ is \eqref{eq:exff}, and $S_0(\boldsymbol{x})$ is \eqref{eq:ex12is}. 

\textbf{Example 1-2.} $\kappa(\boldsymbol{x})$ is \eqref{eq:ex12perm}, $f(S)$ is \eqref{eq:exff}, and $S_0(\boldsymbol{x})$ is \eqref{eq:ex12is}. 

\textbf{Example 1-3.} $\kappa(\boldsymbol{x})$ is \eqref{eq:ex13perm}, $f(S) = S$, and $S_0(\boldsymbol{x})$ is \eqref{eq:ex11is}.

\textbf{Example 1-4.} $\kappa(\boldsymbol{x})$ is \eqref{eq:ex11perm}, $f(S)$ is \eqref{eq:exff}, and $S_0(\boldsymbol{x})$ is \eqref{eq:ex11is}.

For Example 1-3, we can calculate the true solution. Denoting $Y=x_2-x_2^2$ in \eqref{eq:ex13perm} for Example 1-3, method of characteristics gives us the true saturation profile at time $t$ 
\begin{equation} \label{eq:truesat}
S(\boldsymbol{x}, t) = 
\begin{cases}
1, & x_1 < Yt, \\
\frac{1}{1+ (x_1-Yt)^2}, & x_1 \geq Yt.
\end{cases}
\end{equation}

\begin{figure}[ht]
\centering
\includegraphics[height=5.0cm]{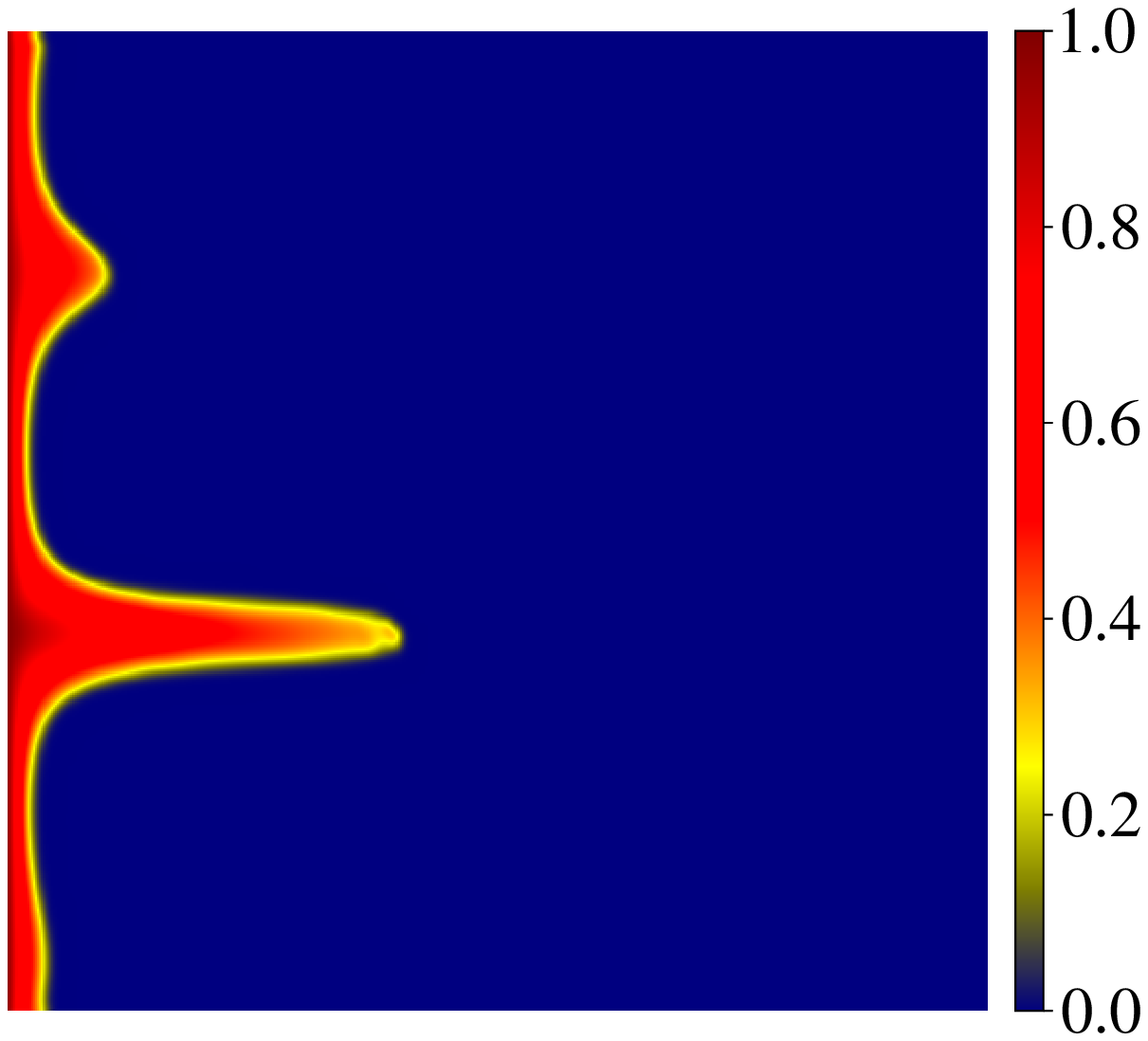}
\includegraphics[height=5.0cm]{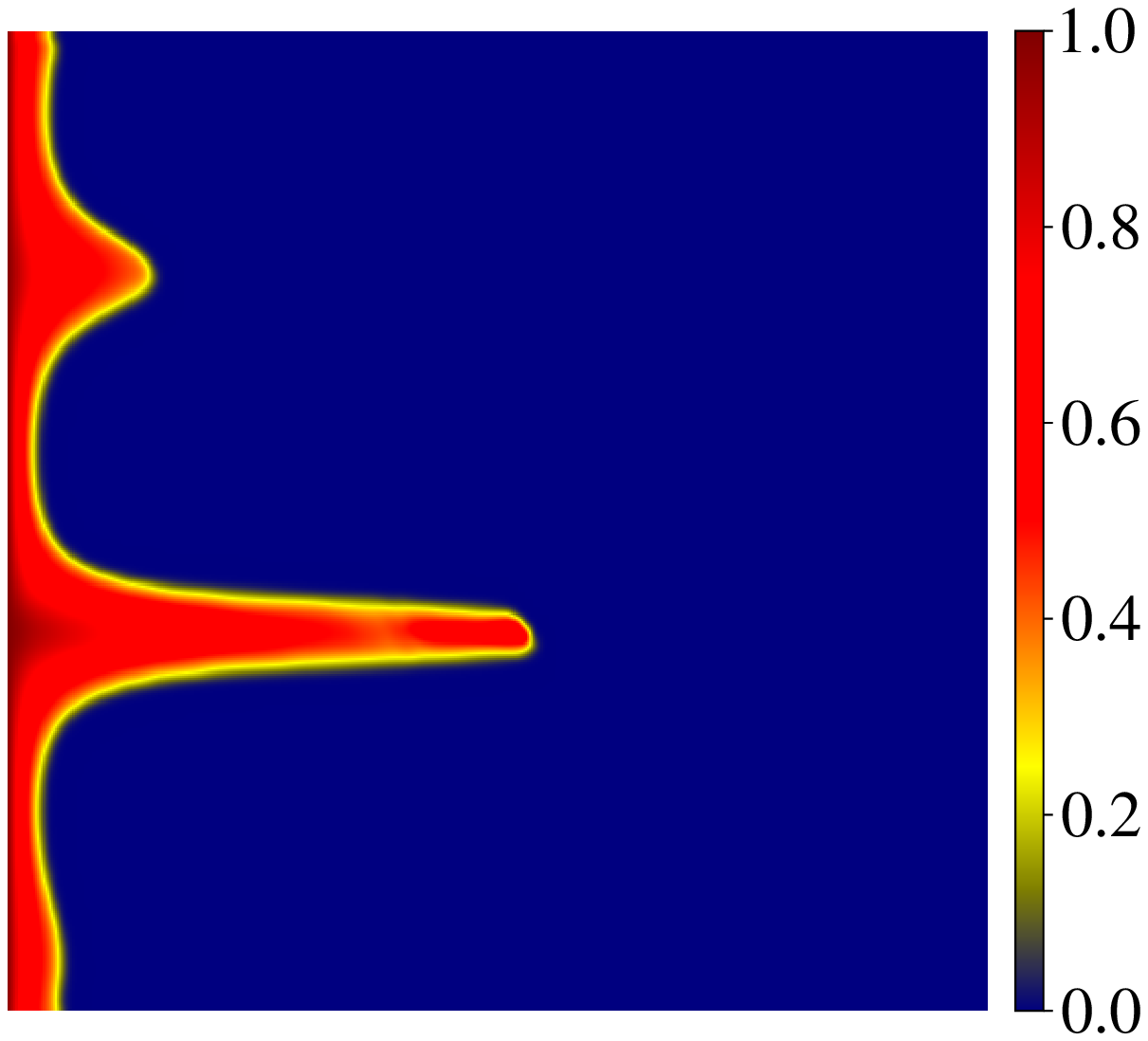}
\includegraphics[height=5.0cm]{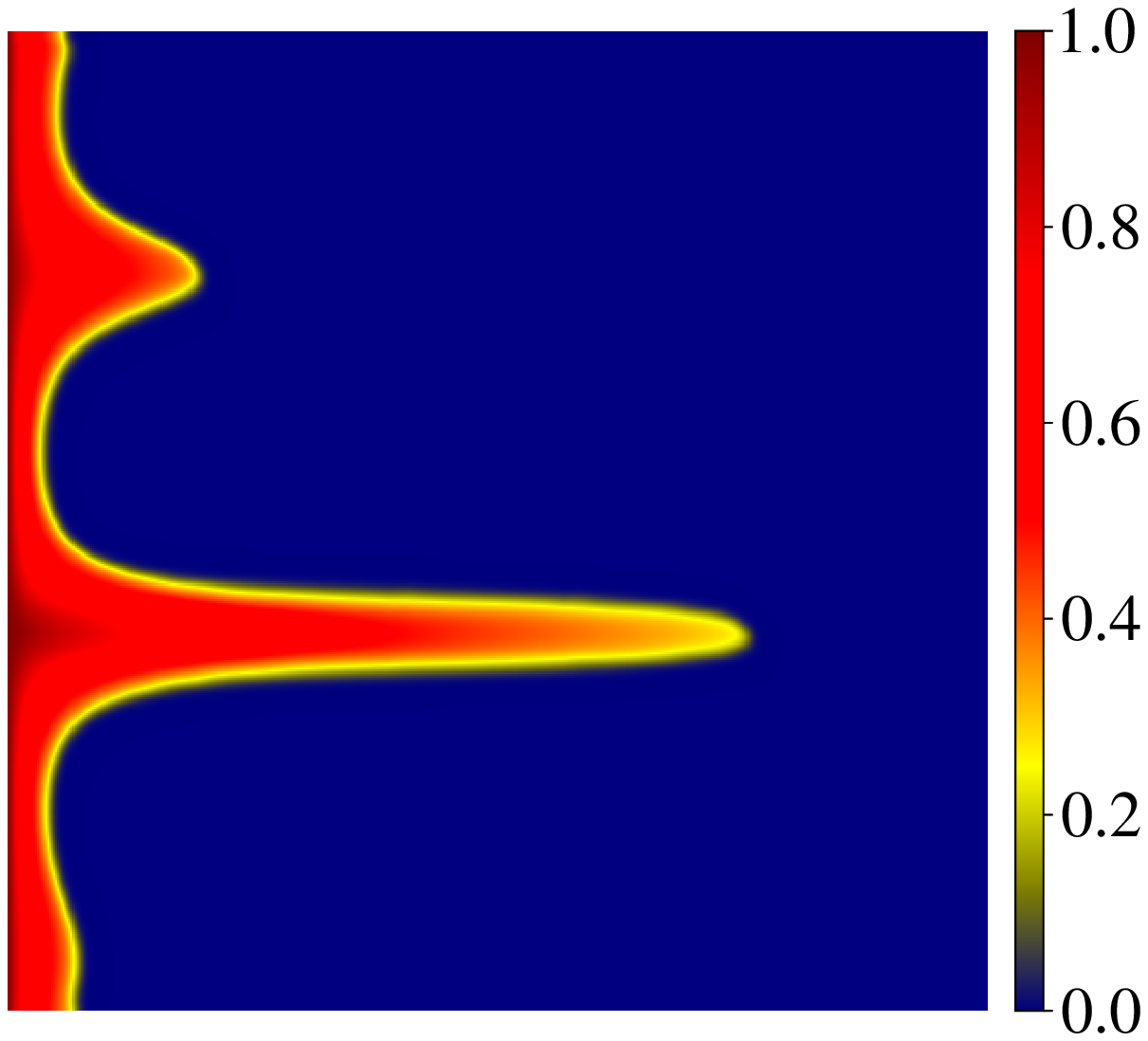}
\caption{Single phase flow simulation for Example 1-2 at time $T=0.001$ (left), $T=0.0015$ (middle), and $T=0.002$ (right).}
\label{fig:ex12}
\end{figure}

Figure \ref{fig:ex12} shows the single phase flow simulation for Example 1-2. The mesh configuration is $128\times 128$ and the number of time steps is $N_t = 500.$ Now we will study the numerical results of these examples.

\subsubsection{Local Conservation Errors} \label{sec:exlce}
\begin{figure}[ht]
\centering
\includegraphics[height=4.0cm]{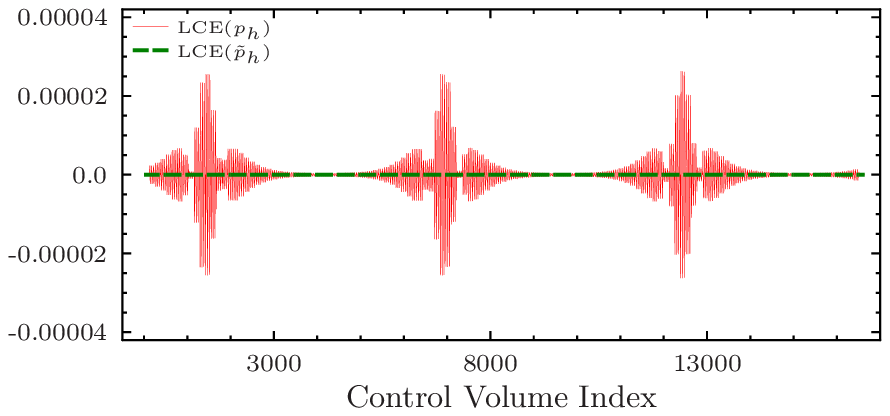}
\includegraphics[height=4.0cm]{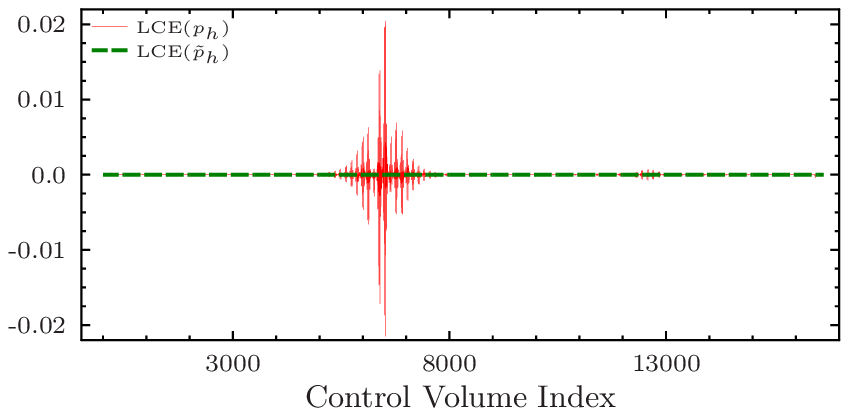} \\
\includegraphics[height=4.0cm]{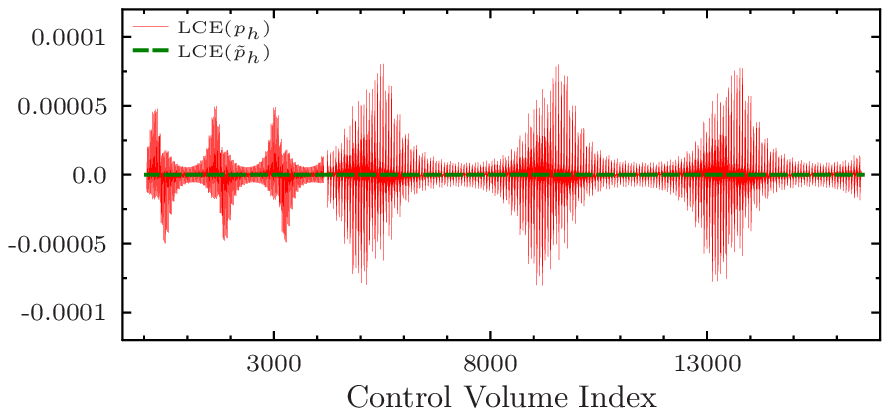}
\includegraphics[height=4.0cm]{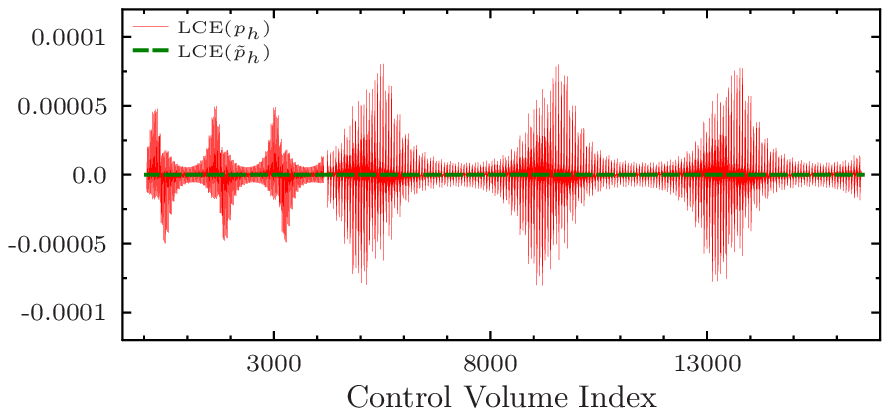} \\
\caption{Local conservation errors of $p_h$ (red line) and $\widetilde p_h$ (dashed-green line) using $V^1_h$ (top) and $V^2_h$ (bottom) for Example 1-1 (left) and Example 1-2 (right).}
\label{fig:ex1lce}
\end{figure}

To verify that the post-processed fluxes satisfy the desired local conservation property \eqref{eq:cvconservation} and thus illustrate the key effect of post-processing, we numerically calculate the local conservation errors defined as
\begin{equation}
\text{LCE}_z(w) = \int_{\partial C^z}  - \kappa(\boldsymbol{x})  \nabla w \cdot \boldsymbol{n}  \ \text{d} l - \int_{C^z} q \ \text{d} \boldsymbol{x}.
\end{equation}
Naturally, $\text{LCE}_z(w)=0$ indicates local conservation property \eqref{eq:cvconservation} is satisfied on the control volume $C^z$. 

The red-line plots in Figure \ref{fig:ex1lce} show the $\text{LCE}_z(p_h)$ solved by linear and quadratic CGFEMs for Example 1-1 and Example 1-2. The mesh configuration for both examples is $128\times 128$ for linear CGFEM and $64\times 64$ for quadratic CGFEM so that they share the same number of total degrees of freedom. The control volume indices in the figure are arranged by indices from vertices of the mesh and then the indices of the degrees of freedom on edges of elements (for quadratic CGFEM). The $\text{LCE}_z(p_h)$ for Example 1-1 is of scale of $10^{-5}$ while for Example 1-2 is of scale of $10^{-1}$. Example 1-2 has a larger scale because the permeability (\ref{eq:ex12perm}) is much more heterogeneous than the permeability (\ref{eq:ex11perm}) of Example 1-1. As we discussed before, these fluxes do not satisfy the local conservation property.

Now by applying the post-processing technique, the LCEs are practically negligible as shown by the green dotted lines in Figure \ref{fig:ex1lce}. To be specific, these errors are of scale of $10^{-15}$, which is mainly attributed to the errors in the linear algebra solver, application of numerical integration, and the machine precision. In an ideal situation, these errors should be zeros as discussed in Section \ref{sec:loccons}.

\subsubsection{$H^1$ Semi-norm Errors of the Post-Processed Solutions} \label{sec:h1seminorm}
Section \ref{sec:exlce} confirms that we do obtain locally conservative fluxes via the post-processing technique but it does not tell us how close the post-processed fluxes to the true fluxes. For this reason, we collect the $H^1$ semi-norm errors, which partially demonstrate the convergence property of the post-processed solution. Table \ref{tab:h1err} shows the $H^1$ semi-norm errors for both Example 1-1 and Example 1-2 using linear and quadratic CGFEMs. In the table, $N_{dof}$ refers to the total number of degrees of freedom. The convergence order of the errors corresponding to linear CGFEM is about 1 while 2 for quadratic CGFEM. Columns two and three, four and five share the same total number of degrees of freedom. Specifically, for Example 1-1, we use mesh configurations $40\times 40, 80\times 80, 160\times 160, 320\times 320$, and $640\times 640$ for linear CGFEM and $20\times 20, 40\times 40, 80\times 80, 160\times 160$, and $320\times 320$ for quadratic CGFEM, while for Example 1-2, we use mesh configurations $80\times 80, 160\times 160, 320\times 320$, and $640\times 640$ for linear CGFEM and $40\times 40, 80\times 80, 160\times 160$, and $320\times 320$ for quadratic CGFEM.

\begin{table}[ht]
\centering
\begin{tabular}{c|c|c|c|c|c}
\cline{2-5}
& \multicolumn{2}{c|}{Example 1-1} & \multicolumn{2}{c|}{Example 1-2} \\ \cline{1-5}
\multicolumn{1}{|c|}{$N_{dof}$} & $V^1_h$ & $V^2_h$  & $V^1_h$ & $V_h^2$  \\ \cline{1-5}
\hline
\multicolumn{1}{|c|}{$ 1681$} & $8.118\times 10^{-2}$ & $3.418\times 10^{-2}$ &  &   \\  \cline{1-5}
\multicolumn{1}{|c|}{$ 6561 $} & $3.991\times 10^{-2}$ & $7.333\times 10^{-3}$ & $7.084\times 10^{-2}$ & $3.505\times 10^{-2}$ \\  \cline{1-5}
\multicolumn{1}{|c|}{$25921$} & $1.986\times 10^{-2}$ & $1.762\times 10^{-3}$ & $3.430\times 10^{-2}$ & $6.625\times 10^{-3}$ \\  \cline{1-5}
\multicolumn{1}{|c|}{$103041$} & $9.918\times 10^{-3}$ & $4.363\times 10^{-4}$ & $1.699\times 10^{-2}$ & $1.468\times 10^{-3}$ \\  \cline{1-5}
\multicolumn{1}{|c|}{$410881$} & $4.957\times 10^{-3}$ & $1.089\times 10^{-4}$ & $8.473\times 10^{-3}$ & $3.578\times 10^{-4}$ \\  \cline{1-5}
\multicolumn{1}{|c|}{order} & 1.008 & 2.066 & 1.020 & 2.202 \\
\hline
\end{tabular}
\caption{$H^1$ semi-norm errors of $\widetilde p_h$ using $V^1_h$ and $V^2_h$ for Example 1-1 and Example 1-2.}
\label{tab:h1err}
\end{table}

Table \ref{tab:h1err} gives an indication that the post-processed solution converges to the true solution as  the mesh is refined. Furthermore, for the same number of degrees of freedom, the post-processing technique applied to quadratic CGFEM tends to provide more accurate locally conservative fluxes than the one applied to linear CGFEM. We will numerically show that it actually yields more accurate saturation profiles in the following section.

\subsubsection{Errors in the Approximate Saturation} 
Sections \ref{sec:exlce} and \ref{sec:h1seminorm} demonstrate the local conservation errors of fluxes and the $H^1$ semi-norm errors of the post-processed solutions. These are behaviors of numerical results for the Darcy's equation before solving the transport equation in the single phase flow simulation. 

Now we apply the post-processed locally conservative fluxes to the transport equation for saturation. We utilize both linear and quadratic CGFEMs for solving Darcy's equation while both upwind and upwind with limiter FVM for solving the transport equation. We expect to obtain more accurate saturation profiles by using quadratic CGFEM and upwind scheme with slope limiter. 

\begin{figure}[ht]
\centering
\includegraphics[height=5.0cm]{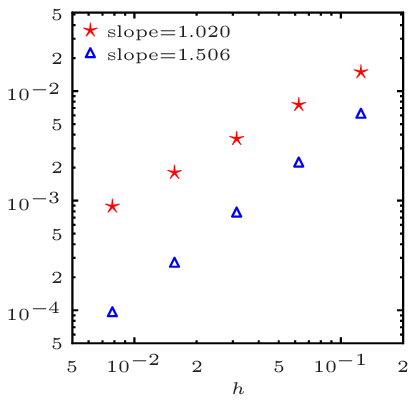} 
\includegraphics[height=5.0cm]{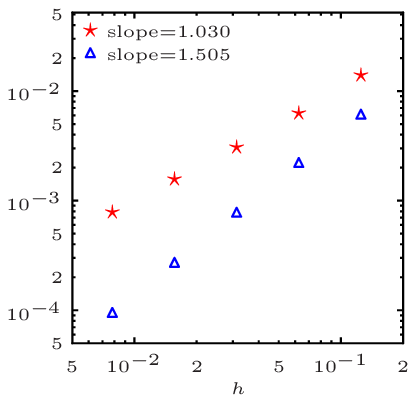}
\caption{Results for Example 1-3. $L^2$-norm errors of saturation using $V^1_h$ (left) and $V^2_h$ (right) with upwind (red star) and upwind with slope limiter (blue triangle). }
\label{fig:ex13err}
\end{figure}

\begin{table}[ht]
\centering
\begin{tabular}{c|c|c|c|c|c}
\cline{2-5}
& \multicolumn{2}{c|}{upwind} & \multicolumn{2}{c|}{upwind with slope limiter} \\ \cline{1-5}
\multicolumn{1}{|c|}{$N_{dof}$} & $V^1_h$ & $V^2_h$  & $V^1_h$ & $V_h^2$  \\ \cline{1-5}
\hline
\multicolumn{1}{|c|}{$ 81$} & $1.488\times 10^{-2}$ & $1.392\times 10^{-2}$ & $6.092\times 10^{-3}$ & $5.980\times 10^{-3}$  \\  \cline{1-5}
\multicolumn{1}{|c|}{$ 289 $} & $7.483\times 10^{-3}$ & $6.268\times 10^{-3}$ & $2.187\times 10^{-3}$ & $2.167\times 10^{-3}$ \\  \cline{1-5}
\multicolumn{1}{|c|}{$1089$} & $3.666\times 10^{-3}$ & $3.062\times 10^{-3}$ & $7.647\times 10^{-4}$ & $7.621\times 10^{-4}$ \\  \cline{1-5}
\multicolumn{1}{|c|}{$4225$} & $1.799\times 10^{-3}$ & $1.567\times 10^{-3}$ & $2.665\times 10^{-4}$ & $2.658\times 10^{-4}$ \\  \cline{1-5}
\multicolumn{1}{|c|}{$16641$} & $8.852\times 10^{-4}$ & $7.836\times 10^{-4}$ & $9.426\times 10^{-5}$ & $9.283\times 10^{-5}$ \\  \cline{1-5}
\multicolumn{1}{|c|}{order} & 1.020 & 1.030 & 1.506 & 1.505 \\
\hline
\end{tabular}
\caption{$L^2$-norm errors of saturation for Example 1-3 at $T=1$ for various scenarios.}
\label{tab:err}
\end{table}

\begin{figure}[ht]
\centering
\includegraphics[height=5.0cm]{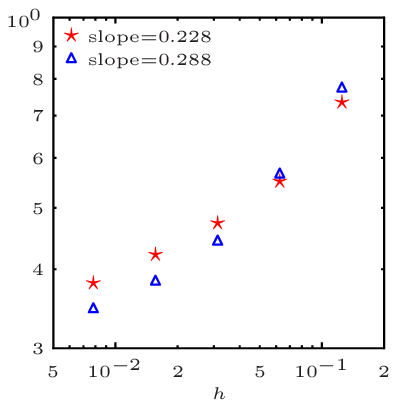} 
\includegraphics[height=5.0cm]{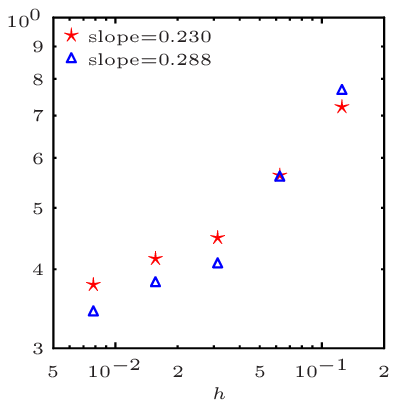}
\caption{Results for Example 1-1. $L^2$-norm errors of saturation using $V^1_h$ (left) and $V^2_h$ (right) with upwind (red star) and upwind with slope limiter (blue triangle). }
\label{fig:ex11err}
\end{figure}
\begin{figure}[ht]
\centering
\includegraphics[height=5.0cm]{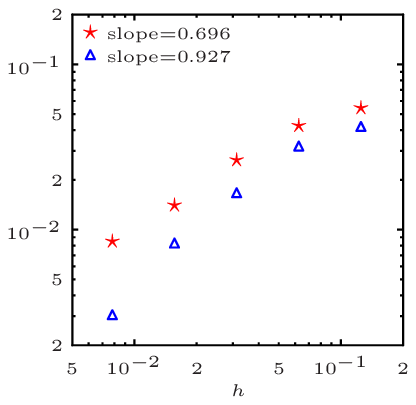} 
\includegraphics[height=5.0cm]{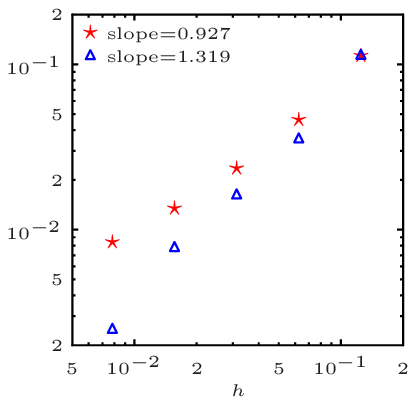}
\caption{Results for Example 1-4. $L^2$-norm errors of saturation using $V^1_h$ (left) and $V^2_h$ (right) with upwind (red star) and upwind with slope limiter (blue triangle). }
\label{fig:ex14err}
\end{figure}

Figure \ref{fig:ex13err} and Table \ref{tab:err} show $L^2$-norm errors of the saturation profile at final time $T=1$ for Example 1-3. The number of time steps is fixed and chosen to be $N_t = 1000$ to satisfy the CFL condition. The mesh configurations are chosen such that both linear and quadratic CGFEM have the same total number of degrees of freedom, $8\times 8, 16\times 16, 32\times 32, 64\times 64$, and $128\times 128$ for linear CGFEM and $4\times 4, 8\times 8, 16\times 16, 32\times 32$, and $64\times 64$ for quadratic CGFEM. For both linear and quadratic CGFEMs, the $L^2$-norm errors are of convergence order 1 when we apply upwind scheme and of convergence order 1.5 when we apply upwind scheme with slope limiter. This on one hand tells us that the error is dominated by the error inherited from upwind schemes rather than the errors inherited from CGFEMs for solving the Darcy's equation, and on the other hand, it confirms that upwind scheme with slope limiter results in a higher order of convergence. However, as we discussed in Section \ref{sec:h1seminorm}, quadratic CGFEM provides more accurate fluxes which should help to improve the saturation profile. That is why we see an improvement of error in $L^2$-norm by using quadratic CGFEM with same number of degrees of freedom as shown in the Table \ref{tab:err}.   In general, these errors are dominated by their contributions from the upwind scheme. This suggests that improvement gained from using the quadratic CGFEM is problem dependent. If we have a higher order scheme for solving the transport equation, we should see a significant improvement of using the higher order CGFEMs. This is a motivation for designing higher order schemes for solving the saturation transport equation.



Figure \ref{fig:ex11err} and Figure \ref{fig:ex14err} show the $L^2$-norm errors of Example 1-1 and Example 1-4. They have similar behavior of convergence. The difference between Example 1-1 and Example 1-4 is on their convergence orders which is attributed to the regularity of their initial saturation profiles. The initial saturation profile for Example 1-1 has a jump discontinuity which results in much lower convergence orders than that of Example 1-4 where initial saturation profile is more smooth. The final time for Example 1-1 and 1-4 is $T=0.05$ and the number of time steps is $N_t = 1000.$ Since there are no analytic solutions for these examples, we used a reference saturation profile obtained from the numerical saturation of the single phase model problem solved by quadratic CGFEM at the same final time with the number of time steps $N_t=2000$ using mesh configurations $256\times 256.$

\subsection{Two-Phase Flow} \label{sec:tpf}
For two-phase flow model, we use $q =0, q_w =0$ in \eqref{eq:tpf} with same boundary conditions posed for the pressure equation in single-phase flow in Section \ref{sec:spf}. We consider the following examples where total mobility function 
$
\lambda(S) = S^2/1 + (1-S)^2/5
$
, $f(S)$ is \eqref{eq:exff}, and 

\textbf{Example 2-1.} $\kappa(\boldsymbol{x})$ is \eqref{eq:ex11perm}, and $S_0(\boldsymbol{x})$ is \eqref{eq:ex12is}. 

\textbf{Example 2-2.} $\kappa(\boldsymbol{x})$ is \eqref{eq:ex11perm}, and $S_0(\boldsymbol{x})$ is \eqref{eq:ex11is}. 

\textbf{Example 2-3.} $\kappa(\boldsymbol{x})$ is \eqref{eq:ex12perm}, and $S_0(\boldsymbol{x})$ is \eqref{eq:ex11is}.

\begin{figure}[ht]
\centering
\includegraphics[height=5.0cm]{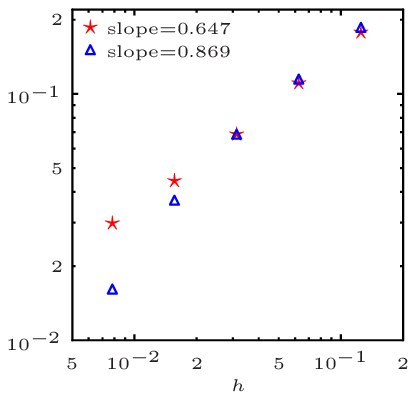} 
\includegraphics[height=5.0cm]{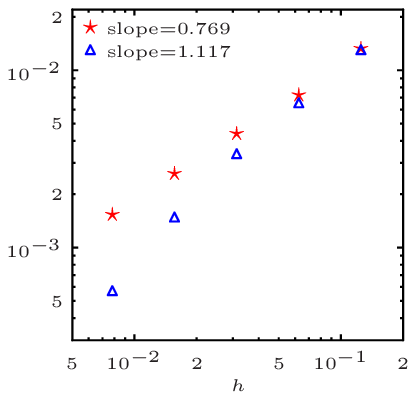} 
\includegraphics[height=5.0cm]{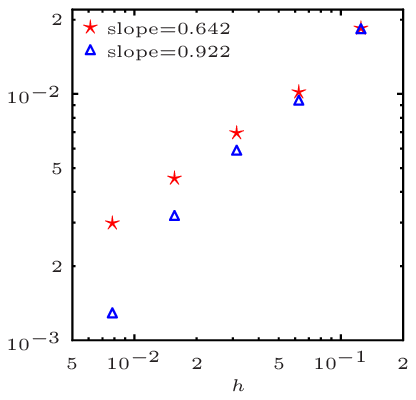} \\
\includegraphics[height=5.0cm]{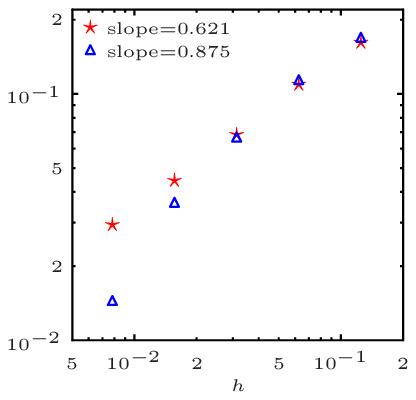}
\includegraphics[height=5.0cm]{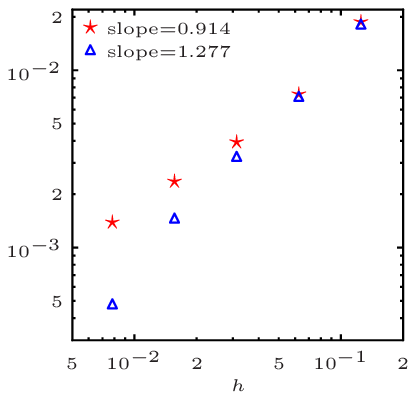}
\includegraphics[height=5.0cm]{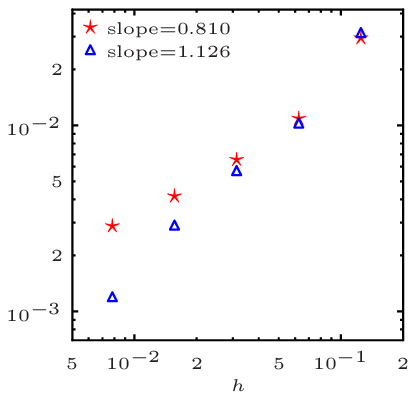}
\caption{$L^2$-norm errors of saturation using $V^1_h$ (top) and $V^2_h$ (bottom) with upwind (red star) and  upwind with slope limiter (blue triangle) for Example 2-1 (left), Example 2-2 (middle), and Example 2-3 (right).}
\label{fig:ex2err}
\end{figure}

Figure \ref{fig:ex2err} shows the numerical results of convergence behaviors of Examples 2-1, 2-2, and 2-3. For Examples 2-1 and 2-2, the final simulation time is $T=0.1$ while for Example 2-3, the final simulation time is $T=0.02.$ For all these examples, we utilized the coarse time step $N_{ct} = 50$ and fine time step $N_{ft} = 100.$ The error collected is the error of the numerical saturation profile to a reference numerical saturation profile at the final time, where the reference numerical saturation profile is the numerical solution at a very fine mesh. In all these examples, we used saturation solved by quadratic CGFEM combined with upwind scheme with slope limiter at the mesh configuration $256 \times 256$ as our reference solution. 

In general, Figure \ref{fig:ex2err} shows that for a given CGFEM, either linear CGFEM or quadratic CGFEM, upwind scheme with slope limiter provides more accurate saturation profile and has a higher order of convergence for the saturation in $L^2$-norm than that of the standard upwind scheme. For a given upwind scheme, either with or without slope limiter, quadratic CGFEM with exact the same degrees of freedom as linear CGFEM provides more accurate control volume boundary fluxes than that of linear CGFEM, especially for finer meshes, which results in a higher order of convergence on saturation in $L^2$-norm. However, the $L^2$-norm error of the saturation is dominated by the schemes applied to solve the transport equation. Hence in this situation, the advantages of using quadratic CGFEM is not optimized.  If a higher order method is provided to solve the saturation transport equation, we would expect a higher order of convergence on saturation in $L^2$-norm. Also, similarly as in the single-phase flow simulation, the initial saturation profile for Example 2-1 has a jump discontinuity which results in lower order of convergence that can be seen in Figure \ref{fig:ex2err}.

\section{Conclusion} \label{sec:conclusion}
We presented a post-processing technique for linear and quadratic CGFEM to obtain locally conservative fluxes. The post-processing technique can be naturally generalized to higher order of CGFEMs and to rectangular meshes. The CGFEM is utilized to solve the Darcy equation in the two-phase flow problem and the post-processing of the CGFEM solutions ensures the locally conservative numerical velocity, which is crucial for robustness of the solver for saturation transport equation in the sense of generating stable and positivity preserving solution.

In this paper, the saturation transport equation is solved by FVM combined with upwind schemes. We presented the widely used upwind scheme as well as developed an upwind scheme with slope limiter for triangular meshes. The analysis of the upwind scheme with slope limiter for triangular meshes is still ongoing. We showed numerically that the upwind scheme with slope limiter improves the convergence property when applied to two-phase flow model problem, especially in the case where the initial saturation profile has sufficient regularity. Moreover, the $L^2$-norm error of the saturation is dominated by the method for solving the saturation transport equation. This is a motivation for designing a method with higher convergence order for the transport equation and we will further investigate this topic in the future.


\bibliographystyle{siam}
\bibliography{tpf}

\begin{thebibliography}{10}

\bibitem{brenner2008mathematical}
{\sc S.~C. Brenner and L.~R. Scott}, {\em The mathematical theory of finite
  element methods}, vol.~15 of Texts in Applied Mathematics, Springer, New
  York, third~ed., 2008.

\bibitem{bush2013application}
{\sc L.~Bush and V.~Ginting}, {\em On the application of the continuous
  {G}alerkin finite element method for conservation problems}, SIAM J. Sci.
  Comput., 35 (2013), pp.~A2953--A2975.

\bibitem{bush2015locally}
{\sc L.~Bush, D.~Quanling, and V.~Ginting}, {\em A locally conservative stress
  recovery technique for continuous galerkin fem in linear elasticity},
  Computer Methods in Applied Mechanics and Engineering, 286 (2015),
  pp.~354--372.

\bibitem{chavent1991unified}
{\sc G.~Chavent and J.~Roberts}, {\em A unified physical presentation of mixed,
  mixed-hybrid finite elements and standard finite difference approximations
  for the determination of velocities in waterflow problems}, Advances in Water
  Resources, 14 (1991), pp.~329--348.

\bibitem{cockburn2007locally}
{\sc B.~Cockburn, J.~Gopalakrishnan, and H.~Wang}, {\em Locally conservative
  fluxes for the continuous {G}alerkin method}, SIAM J. Numer. Anal., 45
  (2007), pp.~1742--1776 (electronic).

\bibitem{cui2007combined}
{\sc M.~Cui}, {\em A combined mixed and discontinuous galerkin method for
  compressible miscible displacement problem in porous media}, Journal of
  computational and applied mathematics, 198 (2007), pp.~19--34.

\bibitem{deng_construction}
{\sc Q.~Deng and V.~Ginting}, {\em Construction of locally conservative fluxes
  for high order continuous galerkin finite element methods}, submitted (2015).

\bibitem{deng2015construction}
{\sc Q.~Deng and V.~Ginting}, {\em {Construction of locally conservative fluxes
  for the {SUPG} method}}, Numerical Methods for Partial Differential
  Equations, 31 (2015), pp.~1971--1944.

\bibitem{deng_locally}
{\sc Q.~Deng, V.~Ginting, B.~McCaskill, and P.~Torsu}, {\em A locally
  conservative stabilized continuous galerkin finite element method for
  multiphase flow in poroelastic subsurfaces}, submitted (2015).

\bibitem{du2010efficient}
{\sc C.~Du and D.~Liang}, {\em An efficient s-ddm iterative approach for
  compressible contamination fluid flows in porous media}, Journal of
  Computational Physics, 229 (2010), pp.~4501--4521.

\bibitem{efendiev2012robust}
{\sc Y.~Efendiev, J.~Galvis, R.~Lazarov, S.~Margenov, and J.~Ren}, {\em Robust
  two-level domain decomposition preconditioners for high-contrast anisotropic
  flows in multiscale media}, Computational Methods in Applied Mathematics
  Comput. Methods Appl. Math., 12 (2012), pp.~415--436.

\bibitem{epshteyn2007fully}
{\sc Y.~Epshteyn and B.~Rivi{\`e}re}, {\em Fully implicit discontinuous finite
  element methods for two-phase flow}, Applied Numerical Mathematics, 57
  (2007), pp.~383--401.

\bibitem{epshteyn2009analysis}
{\sc Y.~Epshteyn and B.~Rivi{\`e}re}, {\em Analysis of hp discontinuous
  galerkin methods for incompressible two-phase flow}, Journal of Computational
  and Applied Mathematics, 225 (2009), pp.~487--509.

\bibitem{evans2010partial}
{\sc L.~C. Evans}, {\em Partial differential equations}, vol.~19 of Graduate
  Studies in Mathematics, American Mathematical Society, Providence, RI,
  second~ed., 2010.

\bibitem{ewing1983mathematics}
{\sc R.~E. Ewing}, {\em The mathematics of reservoir simulation}, Research
  Frontiers in Appl. Math., SIAM, Philadelphia,  (1983).

\bibitem{forsyth1991control}
{\sc P.~A. Forsyth}, {\em A control volume finite element approach to napl
  groundwater contamination}, SIAM Journal on Scientific and Statistical
  Computing, 12 (1991), pp.~1029--1057.

\bibitem{fuvcik2011discontinous}
{\sc R.~Fu{\v{c}}{\'\i}k and J.~Miky{\v{s}}ka}, {\em Discontinous galerkin and
  mixed-hybrid finite element approach to two-phase flow in heterogeneous
  porous media with different capillary pressures}, Procedia Computer Science,
  4 (2011), pp.~908--917.

\bibitem{ginting2015application}
{\sc V.~Ginting, G.~Lin, and J.~Liu}, {\em On application of the weak galerkin
  finite element method to a two-phase model for subsurface flow}, Journal of
  Scientific Computing,  (2015), pp.~1--15.

\bibitem{gmeiner2014local}
{\sc B.~Gmeiner, C.~Waluga, and B.~Wohlmuth}, {\em Local mass-corrections for
  continuous pressure approximations of incompressible flow}, SIAM J. Numer.
  Anal., 52 (2014), pp.~2931--2956.

\bibitem{hoteit2008efficient}
{\sc H.~Hoteit and A.~Firoozabadi}, {\em An efficient numerical model for
  incompressible two-phase flow in fractured media}, Advances in Water
  Resources, 31 (2008), pp.~891--905.

\bibitem{hughes2000continuous}
{\sc T.~J.~R. Hughes, G.~Engel, L.~Mazzei, and M.~G. Larson}, {\em The
  continuous {G}alerkin method is locally conservative}, J. Comput. Phys., 163
  (2000), pp.~467--488.

\bibitem{kou2014analysis}
{\sc J.~Kou and S.~Sun}, {\em Analysis of a combined mixed finite element and
  discontinuous galerkin method for incompressible two-phase flow in porous
  media}, Mathematical Methods in the Applied Sciences, 37 (2014),
  pp.~962--982.

\bibitem{leveque2002finite}
{\sc R.~J. LeVeque}, {\em Finite volume methods for hyperbolic problems},
  Cambridge Texts in Applied Mathematics, Cambridge University Press,
  Cambridge, 2002.

\bibitem{michel2003finite}
{\sc A.~Michel}, {\em A finite volume scheme for two-phase immiscible flow in
  porous media}, SIAM Journal on Numerical Analysis, 41 (2003), pp.~1301--1317.

\bibitem{monteagudo2004control}
{\sc J.~Monteagudo and A.~Firoozabadi}, {\em Control-volume method for
  numerical simulation of two-phase immiscible flow in two-and
  three-dimensional discrete-fractured media}, Water Resources Research, 40
  (2004).

\bibitem{murad2013new}
{\sc M.~A. Murad, M.~Borges, J.~A. Obreg{\'{o}}n, and M.~Correa}, {\em {A new
  locally conservative numerical method for two-phase flow in heterogeneous
  poroelastic media}}, Computers and Geotechnics, 48 (2013), pp.~192--207.

\bibitem{nayagum2004modelling}
{\sc D.~Nayagum, G.~Sch{\"a}fer, and R.~Mos{\'e}}, {\em Modelling two-phase
  incompressible flow in porous media using mixed hybrid and discontinuous
  finite elements}, Computational Geosciences, 8 (2004), pp.~49--73.

\bibitem{nithiarasu2004simple}
{\sc P.~Nithiarasu}, {\em A simple locally conservative galerkin (lcg)
  finite-element method for transient conservation equations}, Numerical Heat
  Transfer, Part B: Fundamentals, 46 (2004), pp.~357--370.

\bibitem{reichenberger2006mixed}
{\sc V.~Reichenberger, H.~Jakobs, P.~Bastian, and R.~Helmig}, {\em A
  mixed-dimensional finite volume method for two-phase flow in fractured porous
  media}, Advances in Water Resources, 29 (2006), pp.~1020--1036.

\bibitem{russell1983finite}
{\sc T.~F. Russell and M.~F. Wheeler}, {\em Finite element and finite
  difference methods for continuous flows in porous media}, The mathematics of
  reservoir simulation, 1 (1983), pp.~35--106.

\bibitem{sun2009locally}
{\sc S.~Sun and J.~Liu}, {\em A locally conservative finite element method
  based on piecewise constant enrichment of the continuous {G}alerkin method},
  SIAM J. Sci. Comput., 31 (2009), pp.~2528--2548.

\bibitem{taflove2000computational}
{\sc A.~Taflove and S.~C. Hagness}, {\em Computational electrodynamics: the
  finite-difference time-domain method}, Artech House, Inc., Boston, MA,
  second~ed., 2000.
\newblock With 1 CD-ROM (Windows).

\bibitem{thomas2008element}
{\sc C.~Thomas and P.~Nithiarasu}, {\em An element-wise, locally conservative
  galerkin (lcg) method for solving diffusion and convection--diffusion
  problems}, International journal for numerical methods in engineering, 73
  (2008), pp.~642--664.

\bibitem{thomas2008locally}
{\sc C.~Thomas, P.~Nithiarasu, and R.~Bevan}, {\em The locally conservative
  galerkin (lcg) method for solving the incompressible navier--stokes
  equations}, International journal for numerical methods in fluids, 57 (2008),
  pp.~1771--1792.

\bibitem{thomas2013numerical}
{\sc J.~W. Thomas}, {\em Numerical partial differential equations: conservation
  laws and elliptic equations}, vol.~33, Springer Science \& Business Media,
  2013.

\bibitem{wang2000approximation}
{\sc H.~Wang, D.~Liang, R.~E. Ewing, S.~L. Lyons, and G.~Qin}, {\em An
  approximation to miscible fluid flows in porous media with point sources and
  sinks by an eulerian--lagrangian localized adjoint method and mixed finite
  element methods}, SIAM Journal on Scientific Computing, 22 (2000),
  pp.~561--581.

\bibitem{zhang2013locally}
{\sc N.~Zhang, Z.~Huang, and J.~Yao}, {\em Locally conservative galerkin and
  finite volume methods for two-phase flow in porous media}, Journal of
  Computational Physics, 254 (2013), pp.~39--51.

\end{thebibliography}

\end{document}